\documentclass[10pt]{article}
\usepackage{amsmath}
\usepackage{amsfonts}
\usepackage{amssymb}
\usepackage[amsthm,amsmath,thmmarks]{ntheorem}
\usepackage{graphicx}

\newcommand{\eps}{\varepsilon}

\newcommand{\go}{\rightarrow}

\newcommand{\pt}{\partial_t}

\newcommand{\px}{\partial_x}

\newcommand{\pxx}{\partial_{xx}}

\newcommand{\ee}{\end{equation}}
\newcommand{\be}{\begin{equation}}
\newcommand{\bea}{\begin{eqnarray}}
\newcommand{\eea}{\end{eqnarray}}
\newcommand{\sbea}{\begin{subequations}\begin{eqnarray}}
\newcommand{\seea}{\end{eqnarray}\end{subequations}} 
\newcommand{\ees}{\end{equation*}}
\newcommand{\bes}{\begin{equation*}}
\newcommand{\beas}{\begin{eqnarray*}}
\newcommand{\eeas}{\end{eqnarray*}}

\newcommand{\rf}[1]{(\ref{#1})}

\newcommand{\Real}{\mathbb{R}}

\newtheorem{thm}{Theorem}[section]

\numberwithin{equation}{section}

\title{Weak solutions to lubrication systems describing the evolution of bilayer
thin films}
\author{Sebastian Jachalski\thanks{Weierstrass Institute,
Mohrenstra{\ss}e 39, 10117 Berlin, Germany,  E-Mail: sebastian.jachalski@wias-berlin.de} \and 
Georgy Kitavtsev\thanks{Max Planck Institute for Mathematics in the Sciences, Inselstra{\ss}e 22, 04103 Leipzig, Germany, E-Mail: georgy.kitavtsev@mis.mpg.de} \and Roman Taranets\thanks{Institute of Applied Mathematics and Mechanics of the National Academy of Sciences of Ukraine, Donetsk, 83114 Ukraine, E-Mail: taranets\_r@yahoo.com}}

\begin{document}
\maketitle

\begin{abstract}
The existence of global nonnegative weak solutions is proved for coupled
one-dimensional lubrication systems that describe the evolution of nanoscopic bilayer
thin polymer films that take account of Navier-slip  or no-slip conditions at both
liquid-liquid and liquid-solid interfaces. In addition, in the presence of attractive
van der Waals and repulsive Born intermolecular interactions existence of positive
smooth solutions is shown.
\end{abstract}



\section{Introduction}

During the last decades lubrication theory was successfully applied to modeling of dewetting processes in micro and nanoscopic liquid films on a solid polymer substrates see e.g. \cite{bertozzi2001dewetting,MWW06,ODB97} to name a few, for a review we refer to \cite{craster2009dynamics} and references therein. 
A typical closed-form one-dimensional lubrication equation derived from the underlying
equations for conservation of mass and momentum, together with boundary
conditions for the tangential and normal stresses, as well as the kinematic
condition at the free boundary, impermeability and a slip condition at the
liquid-solid interface has the form:
\be
\pt h= - \px \Big(M(h)\px \left( \pxx h-\Pi(h)\right)\Big), 
\label{NSM}
\ee
where function $h(x,t)$ denotes the height profile for the free surface of the film. The mobility function has the form $M(h)=h^3$ or $M(h)=bh^2$ for the no-slip
or Navier-slip conditions considered at the solid-liquid interface, respectively, where $b>0$ denotes the slip-length parameter.

Recently, this model was generalized to a coupled lubrication system describing evolution of a layered
system of two viscous, immiscible, nanoscopic Newtonian fluids evolving on a solid
substrate \cite{bandyopadhyay2005instability,danov1998stability,JMPW11} and
subsequently analysed in~\cite{barrett2008finite,jachalskiweierstrass,JMPW11,nepomnyashchy2009dynamics,pototsky2005morphology}.
The latter system  can be stated in the form:  
\begin{equation}\label{mod1}
  \begin{split}
   u_t = - \px \left( M_{11} \px p_1 +M_{12}\px p_2\right),\\
   v_t = - \px \left( M_{12} \px p_1 +M_{22}\px p_2\right), 
\end{split}
\end{equation}
where $u(x,t)$ and $v(x,t)$ denote the height of the lower liquid and the
difference between the heights of the upper and lower liquid,
respectively (see Fig. 1). The pressures  $p_1(x,t)$ and $p_2(x,t)$ are defined as
\begin{equation}\label{mod2}
 \begin{split}
  p_1&=(\sigma+1)\px^2 u+\px^2 v-\Pi_1(u),\\ 
 p_2&=\px^2 u+\px^2 v-\Pi_2(v),
 \end{split}
\end{equation}
where $\px^2 u$ and $\px^2 v$ are linearised surface tension terms and potentials $\Pi_1(u)$ and $\Pi_2(v)$ describe the intermolecular
interactions of the bottom liquid with the solid surface and of two liquids with each other,
respectively. The influence of intermolecular interactions is typically due to the competition between 
long-range attractive van der Waals and short-range Born repulsive intermolecular forces,  see~\cite{gennes85,ODB97}. 
In this article we consider two case: the absence of
intermolecular interactions, i.e. $\Pi_k(s)\equiv 0$ for $k=1,2$ and the case when
both van der Waals and  Born intermolecular forces are presented in the form
\begin{align} 
\Pi_k(s)&=\frac{1}{s^n}-\frac{\gamma_k}{s^m},\ (n<m,\ \gamma_1,\gamma_2\ll 
1)\label{mod4}
\end{align}
A typical choice for $(n,m)$ is $(3,12)$ corresponding to the standard
Lennard-Jones potential.

As in the case of the single layer lubrication equation \rf{NSM} the form of mobility matrix
\begin{align*}
M(u,v)=\begin{pmatrix}
	  M_{11}(u,v) && M_{12}(u,v) \\
	  \ && \ \\
	  M_{21}(u,v) && M_{22}(u,v)
        \end{pmatrix}
\end{align*}
depends on the slip conditions considered at the liquid-liquid and liquid-solid interfaces.
In the case of the no-slip at the both interfaces, it has the form
\begin{align}\label{Mns}
M=\frac{1}{\mu}\begin{pmatrix}
	  \frac{1}{3}u^3 && \frac{1}{2}u^2v \\
	  \ && \ \\
	  \frac{1}{2}u^2v && \frac{\mu}{3}v^3+u\ \!v^2
        \end{pmatrix}.
\end{align}
\begin{figure}[t]
\begin{center}
 \includegraphics[width=0.6\textwidth]{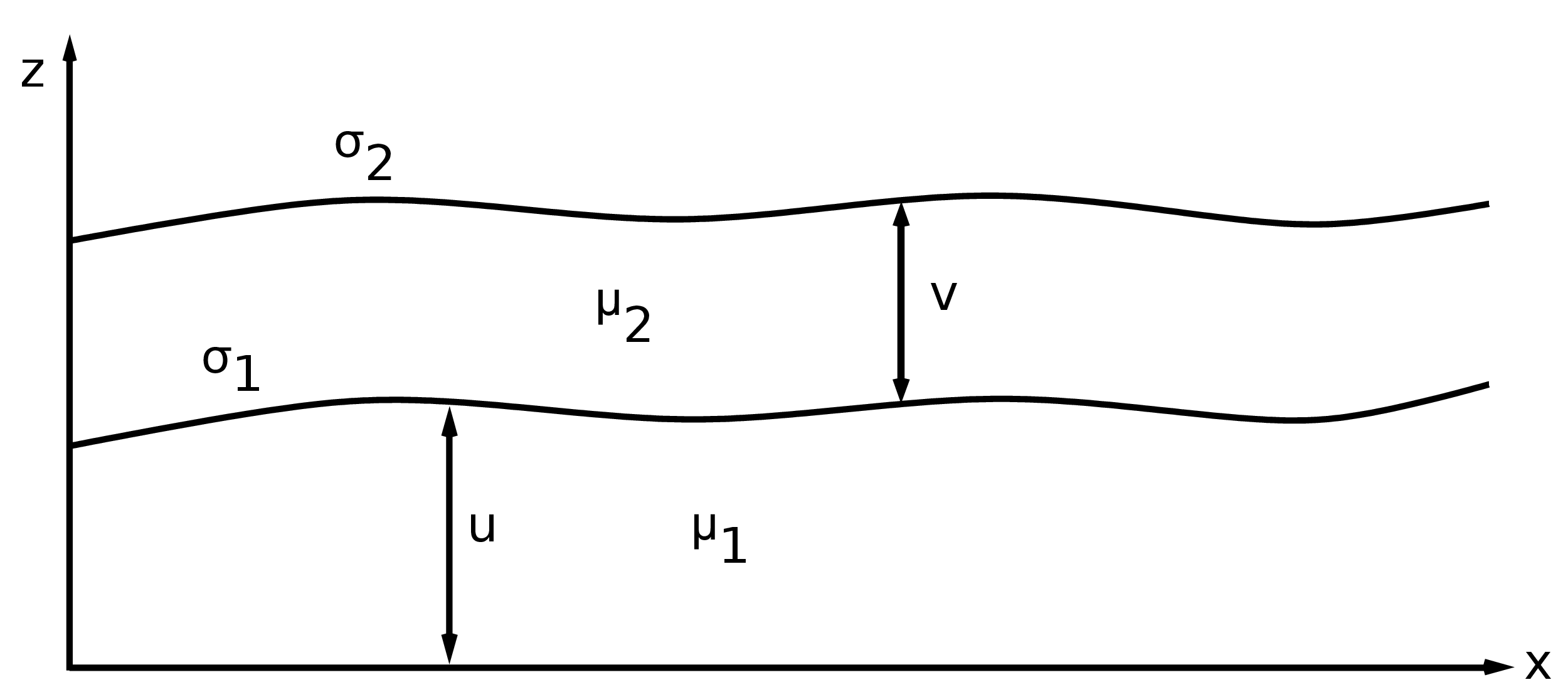}
\caption{Sketch}
\end{center}
\label{fig1}
\end{figure}
The model parameters $\sigma=\sigma_1/\sigma_2$ and $\mu=\mu_1/\mu_2$ in \rf{mod2} and \rf{Mns} are
positive constants which denote the ratios of surface tensions and
viscosities, respectively.
Recently, the lubrication system for the case of Navier-slip conditions considered at both 
liquid-liquid and liquid-solid interfaces was derived in~\cite{JMPW11,pototsky2005morphology}.
The corresponding mobility matrix can be stated in the form
\begin{align*}
M=\frac{1}{\mu}\begin{pmatrix}
	  b_1u^2 && b_1uv \\
	  \ && \ \\
	  b_1uv && b_1v^2+b(\mu+1)v^2
        \end{pmatrix},
\end{align*}
where $b_1>0$ and $b\ge 0$ denote the slip lengths at the solid-liquid and liquid-liquid interfaces, respectively. We will discuss the origin of the different mobilities in section 5.
Note that rescaling time by $b_1$ and introducing the parameter $\alpha:=\frac{b}{b_1}(\mu+1)>0$ the latter matrix can be written in the form
\begin{align}\label{Mis}
M=\frac{1}{\mu}\begin{pmatrix}
	  u^2 && uv \\
	  \ && \ \\
	  uv && (1+\alpha)v^2
        \end{pmatrix}.
\end{align}

In this study we consider the system \rf{mod1} on a time-space domain $Q_T=\Omega\times (0,T)$ where $\Omega=(0,1)$
with boundary conditions
\begin{align}
u_x=u_{xxx}=v_x=v_{xxx}=0\quad\ &\text{for } x\in \partial\Omega\label{basbc}
\end{align}
and the initial functions
\begin{align}
 u(x,0)=u_0(x)\ge 0,\ v(x,0)=v_0(x)\ge 0\quad\ &u,v\in H^1(\Omega).\label{basic}
\end{align}
We consider the system \rf{mod1} both with the mobility matrixes \rf{Mns}
and \rf{Mis}. The system \rf{mod1} can be also generalised to incorporate
presence surfactants or temperature-gradient-caused Marangoni flows (see
e.g. \cite{chugunova2012nonnegative,danov1998stability,nepomnyashchy2009dynamics}).

Starting from the seminal work of  Bernis and Friedman~\cite{bernis1990higher}
existence theory of nonnegative weak solutions for single nonlinear parabolic
equations in the form \rf{NSM} was successfully developed
(see e.g. \cite{BP98,chugunova2010nonnegative} and references there in). 
Note that the system \rf{mod1} inherits from the one-layer lubrication
equation \rf{NSM} the high order and degeneracy as one of the fluid heights
$u$ or $v$ goes to zero.
In contrast to \rf{NSM} there are only few analytical results known about
the system \rf{mod1}. The structure of its stationary
solutions were considered in \cite{jachalskiweierstrass,pototsky2005morphology}.
Existence of nonnegative weak solutions to \rf{mod1} in the no-slip case, i.e. with mobility matrix given by \rf{Mns} was shown recently in
\cite{barrett2008finite} using a finite-element approximation
under a strong assumption on presence of intermolecular potentials of the
form \rf{mod4} between liquid films and between each film and the substrate as
well. 

In this article, we show existence of nonnegative weak solutions to \rf{mod1} with \rf{basbc}--\rf{basic} 
for both no-slip \rf{Mns} and Navier-slip cases \rf{Mis} in the absence of
intermolecular forces. In turn in the presence of intermolecular forces
between liquid films and just between bottom liquid and the substrate as in \rf{mod2} we show that
the observed weak solution becomes positive and smooth.

In our approach we extend ideas introduced in
\cite{bernis1990higher} for the single lubrication equation of the form
\rf{NSM} to the system \rf{mod1}. However, the extension is not straight forward. There are new challenges since the mobility matrix degenerates in more than one way. Beside the case that $u$ and $v$ vanish simultaneously, also the mobility matrix $M$ may degenerates if either $u$ or $v$ becomes zero while the other does not. All these cases have to be treated very carefully. In section 2  we introduce the
corresponding regularized version of the system \rf{mod1} for the no-slip case, with mobility matrix \rf{Mns}. By deriving the energy
dissipation and corresponding a priori estimates, using the theory of
uniformly parabolic systems (see \cite{euidelʹman1969parabolic}), we show global existence
of smooth solutions to the regularised problem. Furthermore, we show that the
latter converge to a suitably defined weak solution of the original system
\rf{mod1}. Notice that these results are independent of the presence of
intermolecular forces in the equations.

In section 3 we prove
nonnegativity of thus obtained weak solutions in the case when intermolecular
forces are absent. In the presence of intermolecular forces as in \rf{mod1} the weak solutions
turn out to be positive and smooth.
Our approach for proving the nonnegativity is based on a definition of suitable
analogs of Bernis and Friedman entropies for functions $u$ and $v$ and showing
their combined dissipation.

Moreover, in section 4 we show global existence of nonnegative weak solutions to \rf{mod1} in the Navier-slip case \rf{Mis}.

\section{Existence of weak solutions in the no-slip case}

In this section we consider the system \rf{mod1}
without intermolecular interactions, i.e. with $p_1,p_2$ in \rf{mod2} given by
\begin{align}
 p_1=(\sigma+1)u_{xx}+v_{xx},\ \ p_2=u_{xx}+v_{xx};
 \label{pot}
\end{align}
and with the no-slip mobility matrix
\begin{align*}
M=\frac{1}{\mu}\begin{pmatrix}
        \frac{1}{3}|u|^3 && \frac{1}{2}|u|^2|v| \\
	\ && \ \\
	\frac{1}{2}|u|^2|v| && \frac{\mu}{3}|v|^3+|u||v|^2
        \end{pmatrix}.
\end{align*}
Notice that we replaced every $u$ and $v$ in the mobility matrix \rf{Mns} by their absolute values to ensure that the latter is positive semidefinite.
We  will prove existence of global weak solutions to \rf{mod1}, \rf{pot} considered with boundary and initial conditions \rf{basbc}--\rf{basic}.
We begin our analysis with introduction of a regularised version of \rf{mod1}, \rf{basic} and derivation of a priori estimates for its solutions.

\subsection{Regularised system and a priori estimates}

Since \rf{mod1} is degenerate at $u=0$ and $v=0$, we approximate it by a family of non-degenerate equations
\begin{align}
   \begin{matrix} u_{t}+((M_{11}+\eps) p_{1,x}+M_{12}p_{2,x})_x=0 \\
   v_{t}+(M_{21} p_{1,x}+(M_{22}+\eps)p_{2,x})_x=0\end{matrix} \quad \text{in } Q_{T},\label{regequa2}
\end{align}
where $\eps>0$ is arbitrary. Note that the regularised mobility matrix is positive definite for all $u$ and $v$. Correspondingly the system \rf{regequa2}
is uniformly parabolic in Petrovskii sense (see \cite{euidelʹman1969parabolic} for definition).
Furthermore, we approximate $u_0$ and $v_0$ in the $H^1(\Omega)$-norm by $C^{4+\alpha}$ functions $u_{0\eps}$ and $v_{0\eps}$ satisfying \rf{basbc},
\begin{align}\label{dem}
 u_{0\eps}(x)&\geq u_{0}(x)\text{ and }v_{0\eps}(x)\geq v_{0}(x)\text{ for }x\in\Omega,
\end{align}
and replace \rf{basic} by
\begin{align}\label{regic}
 u(x,0)&=u_{0\eps}(x),\ v(x,0)=v_{0\eps}(x).
\end{align}
By \cite[Theorem. 6.3, p.302]{euidelʹman1969parabolic} the system \rf{regequa2} considered with  \rf{basbc}, \rf{regic} has a unique local solution $(u_{\eps},v_{\eps})$ in $Q_\tau$ for some small $\tau=\tau(\eps)>0$.

Everywhere below in this article we denote by $C$ positive constants independent of $\eps$ which may vary from line to line. Let us also introduce notations 
\begin{align}
 M_\eps&=M(u_{\eps},v_{\eps})\nonumber\\
 p_{1,\eps}&=(\sigma+1)u_{\eps,xx}+v_{\eps,xx},\ \text{and}\ p_{2,\eps}=u_{\eps,xx}+v_{\eps,xx}.
\label{not}
\end{align}
\textbf{A priori estimates}\newline
Let us define an energy (Lyapunov) functional for the system \rf{mod1} coupled with \rf{basbc} as
\begin{align}\label{fte}
 E(u_\eps,v_\eps)=\int\limits_\Omega \left[\sigma u_{\eps,x}^2+(u_{\eps,x}+v_{\eps,x})^2\right]dx
\end{align}
Indeed, differentiating the latter in time along solutions of \rf{mod1} with \rf{basbc} one obtains the corresponding energy equality
\begin{equation}\label{dte}
\begin{split}
 \frac{1}{2}\frac{d}{dt}E(u_\eps,v_\eps)+\int\limits_\Omega ( M_{11\eps}p_{1\eps,x}^2+2 M_{12\eps}p_{1\eps,x}p_{2\eps,x}+M_{22\eps}p_{2\eps,x}^2)dx\\
+\eps\int\limits_\Omega \left(p_{1\eps,x}^2+p_{2\eps,x}^2\right)dx=0.
\end{split}
\end{equation}
Note that the second term in \rf{dte} is nonnegative since  $M_\eps$ is positive semidefinite.
By the approximation properties of $u_{0\eps},v_{0\eps}$ one has
\begin{equation}\label{apri4}
\begin{split}
 \int\limits_\Omega u_{0\eps,x}^2dx\leq(1+\eta(\eps))\int\limits_\Omega u_{0,x}^2dx,\\
 \int\limits_\Omega v_{0\eps,x}^2dx\leq(1+\eta(\eps))\int\limits_\Omega v_{0,x}^2dx,,
\end{split}
\end{equation}
where $\eta(\eps)\rightarrow 0$ if $\eps \rightarrow 0$ and therefore $E(u_{0\eps},v_{0\eps})\leq C$ holds. 
This together with \rf{dte} imply the following a priori estimates:
\begin{align}\label{h1}
 \displaystyle \sup_{t\in(0,\tau)}\int\limits_\Omega u_{\eps,x}^2dx\leq C,\ \sup_{t\in(0,\tau)}\int\limits_\Omega u_{\eps,x}^2dx\leq C,
\end{align}
and
\begin{align}
 \iint\limits_ {Q_{\tau}} ( M_{11\eps}p_{1\eps,x}^2&+2 M_{12\eps}p_{1\eps,x}p_{2\eps,x}+M_{22\eps}p_{2\eps,x}^2)dx\leq C,\label{impbound}\\
&\eps\iint\limits_{Q_{\tau}}\left(p_{1\eps,x}^2+p_{2\eps,x}^2\right)dx\leq C.\label{epsp}
\end{align}
Integrating \rf{regequa2} in time we deduce the conservation of mass law
\begin{align}\label{apri3}
 \int\limits_\Omega u_{\eps,x}(x,t)dx=\int\limits_\Omega u_{0\eps,x}dx,\ \int\limits_\Omega v_{\eps,x}(x,t)dx=\int\limits_\Omega v_{0\eps,x}dx
\end{align}
for all $t\in (0,\tau)$.
Using this, \rf{apri4}, Poincare's inequality and the Sobolev embedding theorem
\begin{align*}
 H^1(\Omega)\subset C^{0,\frac{1}{2}}(\bar{\Omega})
\end{align*}
one obtains
\begin{align}\label{H12}
 \displaystyle ||u_{\eps}(\,.\,,t)||_{C^{0,\frac{1}{2}}(\bar{\Omega})} \leq C,\ ||v_{\eps}(\,.\,,t)||_{C^{0,\frac{1}{2}}(\bar{\Omega})} \leq C.
\end{align}

Next, we obtain uniform H\"older estimates for $u_\eps$ and $v_\eps$ in time. 
Let us introduce functions
\begin{align}
J_{1,\eps}=M_{11\eps}p_{1\eps,x}+M_{12\eps}p_{2\eps,x}\ \ \text{and} \ \ J_{2,\eps}=M_{21\eps}p_{1\eps,x}+M_{22\eps}p_{2\eps,x}.\nonumber 
\end{align}
Observe that for every $t\in (0,\tau)$
 \begin{align}\label{apri7a}
&\iint\limits_{Q_t} J_{1,\eps}^2 dxdt\nonumber\\
&\leq C \iint\limits_{Q_t} M_{11\eps}\left(M_{11\eps}p_{1\eps,x}^2
+2M_{12\eps} p_{1\eps,x}p_{2\eps,x}+M_{22\eps}p_{2\eps,x}^2\right) dxdt\nonumber\\
&\leq C,
\end{align}
where we use $M_{12\eps}^2\leq M_{11\eps}M_{22\eps}$, \rf{impbound} and \rf{H12}. Analogously,
\begin{align}\label{apri7b}
 \iint\limits_{Q_t}  J_{2 ,\eps}^2dxdt\leq C
\end{align}
holds.
Now, using \rf{apri7a}--\rf{apri7b} and the relations
\begin{align}
\iint\limits_{Q_t} u_\eps\phi_t=-\iint\limits_{Q_t} J_{1,\eps}\phi_x\ \ \text{and} \ \ \iint\limits_{Q_t} v_\eps\phi_t=-\iint\limits_{Q_t} J_{2,\eps}\phi_x\nonumber 
\end{align}
considered with the special test function $\phi$ taken exactly as in the analogous proof for the single layer lubrication equation \rf{NSM}
in \cite[Lemma 2.1]{bernis1990higher} one obtains that for all $x\in\bar{\Omega}$ and $t_1,\,t_2$ in $(0,\tau)$ the following holds
\begin{equation}\label{apri8}
 \begin{split}
|u_{\eps}(x,t_2)-u_{\eps}(x,t_1)|&\leq C|t_2-t_1|^\frac{1}{8},\\
|v_{\eps}(x,t_2)-v_{\eps}(x,t_1)|&\leq C|t_2-t_1|^\frac{1}{8}
\end{split}
\end{equation}
\textbf{Conclusion}\newline
The relations \rf{H12} and \rf{apri8} imply upper bounds on the $C^{\frac{1}{2},\frac{1}{8}}_{x,t}$-norms of $u_{\eps}$ and $v_{\eps}$ in $Q_\tau$, which are independent of $\tau,\eps$. These a priori bounds allows us to conclude that $(u_{\eps},v_{\eps})$ can be extended step-by-step to a solution of \rf{regequa2} considered with \rf{basbc}, \rf{regic} in $Q_{T}$ for any positive $T>0$ (see \cite[Theorem. 9.3, p.316]{euidelʹman1969parabolic}), and that
\begin{equation}\label{apri10}
\begin{split}
 \text{the sequences }&\{u_{\eps}\} \text{ and }  \{v_{\eps}\} \text{ are a uniformly bounded}\\
\text{ and }&\text{equi-continuous families in }\bar{Q}_{T}.
\end{split}
\end{equation}

\subsection{Convergence to global weak solutions}

Here we show that solutions $u_\eps,\,v_\eps$ of the regularised system \rf{regequa2} converge to suitably defined global weak solutions of the initial system \rf{mod1}.
By \rf{apri10}, every sequence $\eps\longrightarrow0$ has a subsequence (for short both not labeled) such that
\begin{align}\label{weak1}
 u_{\eps}\longrightarrow u,\ v_{\eps}\longrightarrow v \quad\text{uniformly in }\bar{Q}_{T}.
\end{align}
Note that, due to uniform bounds \rf{H12} and \rf{apri8},  any such limits $u$ and $v$ can be defined globally in time using a standard Cantor diagonal argument (choosing a sequence $T_n\rightarrow \infty$).
\begin{thm}
Any pair of functions $(u,v)$ obtained as in \rf{weak1} satisfies for any $T>0$ the following properties:
\begin{align}
 &u,v\in C^{1/2,1/8}_{x,t}(\bar{Q}_{T}),\ 
 u,v\in C^{4,1}_{x,t}(P),\label{ass1}\\
 &M_{11}p_{1,x}+M_{12}p_{2,x},\ M_{21}p_{1,x}+M_{22}p_{2,x}\in L^2(P),\\
& |u|^3p_{1,x} 		 \in L^2(R),
  |v|^3p_{2,x}		 \in L^2(S);
\end{align}
where $P=\bar{Q}_{T}\backslash(\{u=0\}\cup\{v=0\}\cup\{t=0\})$, $R=\bar{Q}_{T}\cap\{v=0\}\cap\{|u|>0\}$ and $S=\bar{Q}_{T}\cap\{u=0\}\cap\{|v|>0\}$.
Furthermore, there exists a function $w\in L^2(R)$, such that $(u,v)$ satisfies \rf{mod1} in the following sense:
\begin{align}\label{ass3}
 \iint\limits_{Q_{T}}u\phi_{t}&+\iint\limits_P \left(M_{11} p_{1,x}+M_{12}p_{2,x}\right)\phi_{x} \\&\quad\quad+\iint\limits_{R}\left(\frac{1}{3\mu}|u|^3p_{1,x}+\frac{1}{2\mu}|u|^2w\right)\phi_{x}=0\nonumber\\
  \iint\limits_{Q_{T}}v\phi_{t}&+\iint\limits_P \left(M_{21} p_{1,x}+M_{22}p_{2,x}\right)\phi_{x} +\iint\limits_{S}\frac{1}{3}|v|^3p_{2,x}\phi_{x}=0
 \label{ass4}
\end{align}
for all $\phi\in Lip(\bar{Q}_{T}),\phi=0$ near $t=0$ and $t=T$;
\begin{align}
&u(x,0)=u_{0}(x),\ v(x,0)=v_{0}(x),\quad x\in\bar{\Omega},\label{ass2}\\
&||u(\cdot,t)||_{L^1(\Omega)}=||u_0||_{L^1(\Omega)},\ ||v(\cdot,t)||_{L^1(\Omega)}=||u_0||_{L^1(\Omega)},\label{ass6}\\
&u_{x}(\cdot,t)\rightarrow u_{0,x}\text{ and }v_{x}(\cdot,t)\rightarrow v_{0,x}\text{ strongly in }L^2(\Omega)\text{ as }t\rightarrow0,\label{ass5}
\end{align}
and
\begin{equation}
 \begin{split}
  \label{weak8}
u\text{ and }v &\text{ satisfy \rf{basbc} at all points of the lateral}\\ &\text{ boundary, where }u\neq 0\ \text{and}\ v\neq 0.
 \end{split}
\end{equation}
Finally, the following energy inequality holds
\begin{align}\label{EE}
E(u(\cdot,T),v(\cdot,T))+\iint\limits_P \left(M_{21} p_{1,x}^2+M_{22}p_{2,x}^2\right)\le E(u_0,v_0).
\end{align}
\end{thm}
\begin{proof}
 By the properties of the constructed solutions $u_\eps$, $v_\eps$ to the regularised system and their uniform convergence to $u$, $v$ the first assertion in 
\rf{ass1} and also \rf{ass2}--\rf{ass6} follow immediately. Using \rf{impbound}, we observe 
\begin{align}
 \iint\limits_{Q_{T}}(M_{11\eps}p_{1,\eps,x}^2+M_{22\eps}p_{2\eps,x}^2)dxdt&\leq C -2\iint\limits_{Q_{T}} M_{12\eps} p_{1\eps,x}p_{2\eps,x}dxdt\nonumber\\
&= C - \frac{1}{\mu}\iint\limits_{Q_{T}}|u_\eps|^2|v_\eps| p_{1\eps,x}p_{2\eps,x}dxdt.\nonumber
\end{align}
From this applying Young's inequality
\begin{align}
|u_\eps|^2|v_\eps| p_{1\eps,x}p_{2\eps,x} \leq\frac{7}{24}|u_\eps|^3 p_{1\eps,x}^2+\frac{6}{7}|u_\eps||v_\eps|^2p_{2\eps,x}^2\nonumber
\end{align}
one obtains
\begin{align}\label{p1p1}
 \iint\limits_{Q_{T}}(M_{11\eps}p_{1,\eps,x}^2+M_{22\eps}p_{2\eps,x}^2)dxdt&\leq C
\end{align}
and therefore estimates
\begin{align}
 \iint\limits_{Q_{T}}|u_\eps|^3p_{1\eps,x}^2dxdt\leq C,&\ \iint\limits_{Q_{T}}|v_\eps|^3p_{2\eps,x}^2dxdt\leq C,\ \iint\limits_{Q_{T}}|u_\eps||v_\eps|^2p_{2\eps,x}^2dxdt\leq C.
\label{p1p2}
\end{align}
For any $\phi$ as in \rf{ass3} one has
\begin{align}\label{weak9}
 \iint\limits_{Q_{T}}u_{\eps}\phi_{t}dxdt+\iint\limits_{Q_{T}} \left((M_{11\eps}+\eps) p_{1\eps,x}+M_{12\eps}p_{2\eps,x}\right) \phi_{x}dxdt=0,\\
\label{weak10}
\iint\limits_{Q_{T}}v_{\eps}\phi_{t}dxdt+\iint\limits_{Q_{T}} \left(M_{21\eps} p_{1\eps,x}+(M_{22\eps}+\eps)p_{2\eps,x}\right) \phi_{x}dxdt=0.
\end{align}
By \rf{apri7a} and \rf{apri7b} there exist $J_1,J_2\in L^2(Q_{T})$ and a subsequence as $\eps\rightarrow 0$ such that
\begin{align}
 &\left(M_{11\eps} p_{1\eps,x}+M_{12\eps}p_{2\eps,x}\right)\rightharpoonup J_1\nonumber\\
\text{and } \left(M_{21\eps}\!\!\right.&\left.p_{1\eps,x}+M_{22\eps}p_{2\eps,x}\right)\rightharpoonup J_2\ \text{weakly in}\ L^2(Q_{T}). 
\end{align}
Additionally by \rf{epsp}
\begin{align}
 \eps\!\!\iint\limits_{Q_{T}}p_{1\eps,x}\phi_{x}dxdt\rightarrow 0,\ \eps\!\!\iint\limits_{Q_{T}}p_{2\eps,x}\phi_{x}dxdt\rightarrow 0 \text{ as }\eps\rightarrow 0.\nonumber
\end{align}
By regularity theory of uniformly parabolic systems and the uniform H\"older continuity of the $u_{\eps}$ and $v_{\eps}$ we deduce that $u,v\in C^{4,1}_{x,t}(P)$ and
\begin{align}\label{jonp}
 J_1=M_{11} p_{1,x}+M_{12}p_{2,x},\ J_2=M_{21} p_{1,x}+M_{22}p_{2,x}\ \ \text{in}\ \ P.
\end{align}

Next, for a fixed $\delta>0$ define a set $I_{1,\delta}=\{|v|\leq \delta<|u|\}$.
From the estimates \rf{p1p2} 
it follows that there exists $w\in L^2(I_{1,\delta})$ such that  $p_{1\eps,x}\rightharpoonup p_{1,x}$ and $v_\eps p_{2\eps,x}\rightharpoonup w$ 
weakly in $L^2(I_{1,\delta})$ as $\eps\rightarrow 0$. Therefore, one obtains
\begin{align}\label{misc12}
 \iint\limits_{I_{1,\delta}} \left(M_{11\eps} p_{1\eps,x}+M_{12\eps}p_{2\eps,x}\right)\phi_{x} dxdt\rightarrow \iint\limits_{I_{1,\delta}} \left(\frac{1}{3\mu}|u|^3p_{1,x}+\frac{1}{2\mu}|u|^2w\right)\phi_{x}dxdt
\end{align}
as $\eps\rightarrow 0$.
On the other hand, one has the estimate
\begin{align}\label{misc115}
 &\iint\limits_{|u|\leq\delta}(M_{11\eps}p_{1\eps,x}+M_{12\eps}p_{2\eps,x})\phi_{x}dxdt \nonumber\\
&\leq C\left( \ \iint\limits_{|u|\leq\delta} M_{11\eps}\left(M_{11\eps}p_{1\eps,x}^2
+2M_{12\eps} p_{1\eps,x}p_{2\eps,x}+M_{22\eps}p_{2\eps,x}^2\right)dxdt\right)^{1/2}\nonumber\\
&\leq C\delta^{3/2}
\end{align}
Let us decompose the second term in \rf{weak9} as follows
\begin{align}\label{decomp}
 \iint\limits_{|u|>\delta,\ |v|>\delta}\left(M_{11\eps} p_{1\eps,x}+M_{12\eps}p_{2\eps,x}\right)\phi_{x} dxdt
+\iint\limits_{|u|\le\delta}\left(M_{11\eps} p_{1\eps,x}+M_{12\eps}p_{2\eps,x}\right)\phi_{x} dxdt\nonumber\\
+\eps\!\!\iint\limits_{Q_{T}}p_{1\eps,x}\phi_{x}dxdt+\iint\limits_{|u|>\delta,\ |v|\le\delta}\left(M_{11\eps} p_{1\eps,x}+M_{12\eps}p_{2\eps,x}\right)\phi_{x} dxdt
\end{align}
and take the limit $\delta\rightarrow 0$ extracting a proper diagonal subsequence $\eps\rightarrow 0$ as follows. Using \rf{jonp} one has
\begin{align}
\displaystyle &\lim_{\delta\rightarrow 0}\  |\iint\limits_{|u|>\delta,\ |v|>\delta} \left(M_{11\eps} p_{1\eps,x}+M_{12\eps}p_{2\eps,x}\right)\phi_{x} dxdt-
\iint\limits_{P} J_1\phi_{x}dxdt|\le\nonumber\\
 &\lim_{\delta\rightarrow 0}\  |\iint\limits_{|u|>\delta,\ |v|>\delta} \left(M_{11\eps} p_{1\eps,x}+M_{12\eps}p_{2\eps,x}\right)\phi_{x} dxdt-
\iint\limits_{|u|>\delta,\ |v|>\delta} J_1\phi_{x}dxdt|+\nonumber\\
&+\lim_{\delta\rightarrow 0}\ |\iint\limits_{|u|>\delta,\ |v|>\delta}J_1\phi_{x}dxdt-\iint\limits_{P} J_1\phi_{x}dxdt|=0,\nonumber
\end{align}
where in the last line we used that $J_1\phi_{x}$ is a bounded continuous function in $P$.
Furthermore, from \rf{misc12} one obtains
\begin{align}
\displaystyle &\lim_{\delta\rightarrow 0}\  |\iint\limits_{I_{1,\delta}} \left(M_{11\eps} p_{1\eps,x}+M_{12\eps}p_{2\eps,x}\right)\phi_{x} dxdt\nonumber\\
&\quad\quad\quad\quad\quad\quad\quad-
\iint\limits_{v=0,\ |u|>0} \left(\frac{1}{3\mu}|u|^3p_{1,x}+\frac{1}{2\mu}|u|^2w\right)\phi_{x}dxdt|\le\nonumber\\
&\lim_{\delta\rightarrow 0}\  |\iint\limits_{I_{1,\delta}} \left(M_{11\eps} p_{1\eps,x}+M_{12\eps}p_{2\eps,x}\right)\phi_{x} dxdt\nonumber\\
&\quad\quad\quad\quad\quad\quad\quad-
\iint\limits_{I_{1,\delta}} \left(\frac{1}{3\mu}|u|^3p_{1,x}+\frac{1}{2\mu}|u|^2w\right)\phi_{x}dxdt|+\nonumber\\
&+\lim_{\delta\rightarrow 0}\ |\iint\limits_{|v|\le\delta,\ |u|>\delta}\left(\frac{1}{3\mu}|u|^3p_{1,x}+\frac{1}{2\mu}|u|^2w\right)\phi_{x}dxdt\nonumber\\
&\quad\quad\quad\quad\quad\quad\quad
-\iint\limits_{v=0,\ |u|>0} \left(\frac{1}{3\mu}|u|^3p_{1,x}+\frac{1}{2\mu}|u|^2w\right)\phi_{x}dxdt|\le\nonumber\\
&\lim_{\delta\rightarrow 0}\ |\iint\limits_{\delta\ge|v|>0}\left(\frac{1}{3\mu}|u|^3p_{1,x}+\frac{1}{2\mu}|u|^2w\right)\phi_{x}dxdt|\le\nonumber\\
&\lim_{\delta\rightarrow 0}\ C\!\left(\!\left(\ \iint\limits_{\delta\ge|v|>0}|u|^3dxdt\iint\limits_{\delta\ge|v|>0}|u|^3p_{1,x}^2dxdt\right)^\frac{1}{2}\!\!+\!\left(\ \iint\limits_{\delta\ge|v|>0}|u|^4dxdt
\iint\limits_{\delta\ge|v|>0}w^2dxdt\right)^\frac{1}{2}\right)\nonumber\\
 &\le \lim_{\delta\rightarrow 0}\ C\left(\iint\limits_{\ \delta\ge|v|>0}1dxdt\right)^\frac{1}{2}= 0,\nonumber
\end{align}
where in the last inequality we used \rf{ass1} and \rf{p1p2}.
Therefore, the last two estimates together with \rf{misc115} imply that there exists a subsequence $\eps\rightarrow 0$ such that \rf{weak9} converge to \rf{ass3}.

Similarly, defining for a fixed $\delta\ge 0$ the set $I_{2,\delta}=\{|u|\leq\delta<|v|\}$ one can estimate 
\begin{align}\label{misc18}
 \iint\limits_{|v|\leq\delta}(M_{21\eps} p_{1\eps,x}+M_{22\eps}p_{2\eps,x})\phi_{x}dxdt\leq C\delta,
\end{align}
and
\begin{align}
 &\iint\limits_{I_{2,\delta}} \left({M}_{21\eps} p_{1\eps,x}+{M}_{22\eps}p_{2\eps,x}-\frac{1}{3}|v_\eps|^3 p_{2\eps,x}\right)\phi_{x}dxdt\nonumber\\
  &\leq C\left(\iint\limits_{I_{2,\delta}}\left(M_{21\eps}^2 p_{1\eps,x}^2+2M_{21\eps}\left(M_{22\eps}-\frac{1}{3}|v_\eps|^3\right) p_{1\eps,x}p_{2\eps,x}\right.\right.\nonumber\\
&\quad \quad\quad\left.\left.+\left(M_{22\eps}-\frac{1}{3}|v_\eps|^3\right)^2p_{2\eps,x}^2\right)dxdt\right)^{1/2} \nonumber\\
&\leq C\left(\iint\limits_{I_{2,\delta}}\frac{1}{\mu}|u_\eps||v_\eps|^2\left(M_{11\eps} p_{1\eps,x}^2+2M_{21\eps} p_{1\eps,x}p_{2\eps,x}+M_{22\eps}p_{2\eps,x}^2\right)dxdt\right)^{1/2} \nonumber\\
&\leq C\delta^{1/2}.
\end{align}
Moreover, again from \rf{p1p2} it follows that $p_{2\eps,x}\rightharpoonup p_{2,x}$ weakly in $L^2(I_{2,\delta})$ as $\eps\rightarrow 0$ and therefore,
we deduce that
\begin{align}\label{misc17}
 \iint\limits_{I_{2,\delta}} \left(\frac{1}{3}|v_\eps|^3 p_{2\eps,x}\right)\phi_{x} dxdt
\rightarrow \iint\limits_{I_{2,\delta}} \left(\frac{1}{3}|v|^3p_{2,x}\right)\phi_{x}dxdt
\end{align}
as $\eps\go 0$.
Taking the limit $\delta\go 0$ and the corresponding diagonal sequence $\eps\rightarrow 0$ in \rf{weak10} (as was done before for \rf{weak9})  shows that it converges to \rf{ass4} in view of \rf{misc18}--\rf{misc17}.

 To prove \rf{ass5} notice that from $u_{0,\eps}\rightarrow u_{0}$, $v_{0,\eps}\rightarrow v_{0}$ in $H^1(\Omega)$ and \rf{dte} we get
\begin{align*}
\limsup\limits_{t\rightarrow0}\int\limits_\Omega \left(\sigma u_{x}^2(x,t)+(u_{x}+v_x)^2(x,t)\right)dx\leq \int\limits_\Omega \left(\sigma u_{0,x}^2+(u_{0,x}+v_{0,x})^2\right)dx.
\end{align*}
Since also
\begin{align*}
 u_{x}(.,t)\rightarrow u_{0,x} \text{ and }
 v_{x}(.,t)\rightarrow v_{0,x} \text{ weakly in }L^2(\Omega)
\end{align*}
as $t\rightarrow0$, the assertion \rf{ass5} follows.

Finally,  the energy inequality \rf{EE} follows from \rf{dte}, \rf{p1p1} and the standard weakly lower semicontinuity argument. 
The proof of the theorem is complete.
\end{proof}

\section{Nonnegativity of solutions}

In this section we prove that the global weak solutions constructed in the previous section are nonnegative provided the initial data $u_0$ and $v_0$ 
are nonnegative. Furthermore, for the system \rf{pot} considered with intermolecular potentials $\Pi_1(u)$ and $\Pi_2(v)$ as in \rf{mod2}--\rf{mod4} we show
existence of positive smooth solutions.

\subsection{Nonnegativity in the absence of intermolecular forces}

Following ideas of \cite{bernis1990higher}, we define a suitable entropy  in order to show nonnegativity of the weak solutions $u$ and $v$ from the Theorem 2.1,
provided \rf{basic} holds. 

For $n\in\{2,3\}$ we set
\begin{align}\label{G0}
 g_{\eps,n}(s)=-\int\limits_s^A \frac{dr}{(|r|^n+\eps)^{1/2}},\ G_{\eps,n}(s)=-\int\limits_s^A g_{\eps,n}(r)dr
\end{align}
with a constant $A$ such that $A\ge\max\{|u_\eps|,\,|v_\eps|\}$ for all sufficiently small $\eps$. Then one has
\begin{align}
 G_{\eps,n}'(s)=g_{\eps,n}(s),\ G_{\eps,n}''(s)=g_{\eps,n}'(s)=\frac{1}{(|s|^n+\eps)^{1/2}}.\nonumber
\end{align}
Also,
\begin{align}
 g_{\eps,n}(s)\leq 0,\ G_{\eps,n}(s)\geq 0 \text{ if } s\leq A\nonumber
\end{align}
and
\begin{align}\label{G1}
 G_{\eps,n}(s) \leq G_{0}(s)\text{ for all }s\in \Real^1
\end{align}
where $G_{0,n}=\lim_{\eps\rightarrow0} G_{\eps,n}$ such that for $0\leq s\leq A$
\begin{align}\label{G2}
 G_{0}(s)=\left\{\begin{array}{ll} (A-s-s\log\left(\frac{A}{s}\right)), & n=2\\ 2(\sqrt{A}+\sqrt{\frac{1}{A}}s-2\sqrt{s}), & n=3 \end{array}\right. .
\end{align}

Since the structure of \rf{mod1} is not symmetric with respect to $u$ and $v$ we use two entropies: one depending on $u$ and the other on $v$. 
For a fixed $\delta>0$ one has
\begin{align}\label{misc16}
 &\frac{d}{dt}\int\limits_\Omega \left(G_{\eps,3}(u_\eps)+G_{\eps,2}(v_\eps)\right)dx=\int\limits_\Omega \left(G_{\eps,3}'(u_\eps)u_{\eps,t}+G_{\eps,2}'(v_\eps)v_{\eps,t}\right)dx\nonumber\\
&=\int\limits_\Omega \left(G_{\eps,3}''(u_{\eps})u_{\eps,x}((M_{11\eps}+\eps)p_{1\eps,x}+M_{12\eps}p_{2\eps,x})\right)dx\nonumber\\
&\quad+\int\limits_\Omega \left(G_{v\eps}''(v_{\eps})v_{\eps,x}(M_{21\eps}p_{1\eps,x}+(M_{22\eps}+\eps)p_{2\eps,x})\right)dx\nonumber\\
&\leq\frac{\delta}{2} \int\limits_\Omega \left((G_{\eps,3}''(u_{\eps}))^2(M_{11\eps}+\eps)( (M_{11\eps}+\eps)p_{1\eps,x}^2+2 M_{12\eps}p_{1\eps,x}p_{2\eps,x}\right.\nonumber\\
&\quad\left.+(M_{22\eps}+\eps)p_{2\eps,x}^2)\right)dx +\frac{\delta}{2}\int\limits_\Omega \left((G_{v\eps}''(v_{\eps}))^2(M_{22\eps}+\eps)( (M_{11\eps}+\eps)p_{1\eps,x}^2\right.\nonumber\\
&\quad\left.+2 M_{12\eps}p_{1\eps,x}p_{2\eps,x}+(M_{22\eps}+\eps)p_{2\eps,x}^2)\right)dx+\frac{1}{2\delta}\int\limits_\Omega \left(u_{\eps,x}^2+v_{\eps,x}^2 \right)dx.
\end{align}
By definition \rf{G0} it follows that
\begin{align}\label{GC}
 (G_{\eps,3}''(u_{\eps}))^2(M_{11\eps}+\eps)\le C,\ \text{and}\ (G_{\eps,2}''(v_{\eps}))^2(M_{22\eps}+\eps)\le C,
\end{align}
where in the last inequality we used the fact that there exists a constant $C$ such that  $M_{22,\eps}\leq C|v_\eps|^2$ holds. 
Combining \rf{misc16} with the energy inequality \rf{dte} and taking $\delta<1$ one obtains
\begin{align}\label{GC1}
 &\frac{d}{dt}\int\limits_\Omega \left(G_{\eps,3}(u_\eps)+G_{\eps,2}(v_\eps)\right)dx+\frac{1}{2}\frac{d}{dt}E(u_\eps,v_\eps)\\
&\quad+(1-\delta)\int\limits_\Omega \left(M_{11\eps}+\eps)p_{1\eps,x}^2+2 M_{12\eps}p_{1\eps,x}p_{2\eps,x}+(M_{22\eps}+\eps)p_{2\eps,x}^2\right)dx\\
&\leq \frac{1}{2\delta}\int\limits_\Omega \left(u_{\eps,x}^2+v_{\eps,x}^2 \right)dx\leq \frac{C}{2\delta}\int\limits_\Omega \left(\sigma u_{\eps,x}^2+(u_{\eps,x}+v_{\eps,x})^2\right)dx.
\end{align}
This implies using Gronwall inequality
\begin{align*}
\int\limits_\Omega \left(G_{\eps,3}(u_\eps)+G_{\eps,2}(v_\eps)\right)dx \leq \exp\left({\frac{t}{\delta}}\right)\int\limits_\Omega \left(G_{\eps,3}(u_{0\eps})+G_{\eps,2}(v_{0\eps})\right)dx+\frac{1}{2} E(u_{0\eps},v_{0\eps}).
\end{align*}
On the other hand by \rf{basic}, \rf{dem} and \rf{G1}--\rf{G2} one has
\begin{align*}
 \int\limits_\Omega \left(G_{\eps,3}(u_{0\eps})+G_{\eps,2}(v_{0\eps})\right)dx&\leq \int\limits_\Omega \left(G_{0,3}(u_{0\eps})+G_{0,2}(v_{0\eps})\right)dx\\
&\leq \int\limits_\Omega \left(G_{0,3}(u_{0})+G_{0,2}(v_{0})\right)dx\leq C.
\end{align*}
Therefore, the last two estimates imply that for all $t\leq T$
\begin{align}\label{G4}
 \int\limits_\Omega \left(G_{\eps,3}(u_{\eps})+G_{\eps,2}(v_{\eps})\right)dx\leq C.
\end{align}

Finally, we prove the nonnegativity of $u$ and $v$ by contradiction. Assume there is a point $(x_0,t_0)\in Q_T$ such that $u(x_0,t_0)<0$. Since $u_\eps\rightarrow u$ uniformly there exist $\gamma>0,\ \eps_0>0$ such that
\begin{align*}
 u_\eps(x,t_0)<-\gamma\text{ if }\ |x-x_0|<\gamma,\ x\in\bar{\Omega},\ \eps<\eps_0.
\end{align*}
For such $x$
\begin{align*}
 G_{\eps,3}(u_\eps(x,t_0))&=-\int\limits_{u_\eps(x,t_0)}^A g_{\eps,3}(s)ds\geq -\int\limits_{-\gamma}^0 g_{\eps,3}(s)ds\\
&\rightarrow -\int\limits_{-\gamma}^0 g_{0,3}(s)ds \text{ as }\eps\rightarrow 0
\end{align*}
by monotone convergence theorem, where $g_{0,n}(s)=\lim_{\eps\rightarrow0}g_{\eps,n}(s)$. Since by \rf{G0} 
\begin{align}\label{G5}
g_{0,n}(s)=-\infty\ \text{if}\ s<0,\ n\ge 2
\end{align} 
it follows that 
\begin{align*}
\lim_{\eps\rightarrow 0} G_{\eps,3}(u_\eps(x,t_0))=\infty,
\end{align*}
which is a contradiction to \rf{G4}. A completely analogous argument shows  $v\ge 0$ using \rf{G4} and \rf{G5} with $n=2$.

\subsection{The case including intermolecular forces}

In this section we consider the system \rf{mod1}-\rf{mod2} in the presence of the intermolecular forces given as in \rf{mod4}
considered with \rf{basbc} and the initial data satisfying
\begin{align}\label{idp}
 \int\limits_\Omega \left[U_1(u_{0})+U_2\left(v_{0}\right)\right]dx\leq C_1,
\end{align}
where by definition 
\begin{align*}
U_k(s)=-\int\limits_s^\infty \Pi_k(\tau)d\tau.
\end{align*}
\begin{thm}
Assume that $0<n<m$ and $m\geq 3$ in \rf{mod4}. Then a positive smooth solution to \rf{mod1}-\rf{mod4} coupled with \rf{basbc}, \rf{idp}
exists for all $t\in(0,T)$.
\end{thm}
\begin{proof}
Taking a suitable H\"older continuous regularisation of potentials $\Pi_k(s)$ and proceeding as in the section 2 one can show existence
of regularised solutions $u_\eps$ and $v_\eps$ to \rf{regequa2} considered now with \rf{mod1}-\rf{mod4} and \rf{basbc}, \rf{regic}
that satisfy the regularity properties as before.
Note that the energy functional \rf{fte} transforms in this case to
\begin{align*}
 E(u_\eps,v_\eps)=\int\limits_\Omega \left[\sigma u_{\eps,x}^2+(u_{\eps,x}+v_{\eps,x})^2+2U_1(u_{\eps})+2U_2(u_{\eps})\right]dx
\end{align*}
for which the energy inequality \rf{dte} still holds. Therefore, using the fact that $U_k(s),\, k=1,2$ are bounded from below,  and hence
\begin{align}\label{bU}
 -\int\limits_\Omega \left[U_1(u_{\eps})+U_2\left(v_{\eps}\right)\right]dx\leq C
\end{align}
together with \rf{idp} imply again the estimates \rf{h1}--\rf{epsp}.

We show additionally that there exists a constant $\delta$ independent of $\eps$ such that 
\begin{align}\label{m}
 u_{\eps}\geq \delta>0,\ v_{\eps}\geq\delta>0\ \text{hold in}\ Q_T.
\end{align}
Then proceeding to the limit $\eps\rightarrow 0$ as in Theorem 2.1 the smoothness of the positive limits $u$ and $v$ will follow 
from the uniform parabolic theory and \rf{m} and the statement of the theorem will be shown.
Indeed, observe from \rf{dte} that
\begin{align}
\displaystyle \sup_{t\in (0,T)} \int\limits_\Omega \left(U_1(u_{\eps}(\cdot,t))+U_2\left(v_{\eps}(\cdot,t)\right)\right)dx\leq C.
\end{align}
Since $U_2$ is bounded from below one has also
\begin{align*}
\displaystyle \sup_{t\in (0,T)}  \int\limits_\Omega U_1(u_{\eps}(\cdot,t))dx\leq C.
\end{align*}
Let $u_{\eps}(x_0,t)=\min_\Omega u_{\eps}(\cdot,t)$. By H\"older continuity of $u_\eps$ we get
\begin{align*}
 u_{\eps}(x,t)\leq u_{\eps}(x_0,t)+C|x-x_0|^{1/2}.
\end{align*}
Analogously to the proof for the single layer equation \rf{NSM} in \cite{bertozzi2001dewetting} one obtains
for $0<n<m$
\begin{align*}
  C\geq   \int\limits_\Omega U_1(u_{\eps}(\cdot,,t))dx\geq C_{2} \eta(u_{\eps}(x_0,t)) + C_{3},
\end{align*}
where $\eta(s)=-\log s$ for $m=3$, $\eta(s)= s^{3-m}$ for $m>3$. Hence $u_{\eps}(x_0,t)>0$ holds for all $t\in (0,T)$.
The same argument works for $\min_\Omega v_{\eps}(\cdot,t)$. Therefore \rf{m} is true.
\end{proof}

\section{Existence of nonnegative weak solutions in the Navier-slip case}

In this section we show that solutions $u_\eps$ and $v_\eps$ to the regularised system \rf{regequa2}--\rf{regic}
considered now with the Navier-slip mobility matrix 
\begin{align*}
M=\begin{pmatrix}
	  |u|^2+\eps && |u||v| \\
	  \ && \ \\
	  |u||v| && (1+\alpha)|v|^2+\eps
        \end{pmatrix}.
\end{align*}
converge to global nonnegative weak solutions to \rf{mod1}
considered with \rf{Mis}, \rf{pot} and \rf{basbc}--\rf{basic}. Note that the case when intermolecular forces are present i.e. for \rf{mod1} considered with 
\rf{mod2}--\rf{mod4} and \rf{Mis} proceeds then exactly as in Theorem 3.2 for the no-slip case.

The dissipation \rf{dte} of the energy functional \rf{fte} and the corresponding a priori estimates \rf{h1}--\rf{H12} and \rf{apri8}
stay true in the Navier-slip case as well. Therefore, \rf{weak1} holds again up to a subsequence as $\eps\rightarrow 0$. The following theorems that thus
obtained limits $u$ and $v$ are nonnegative global weak solutions.
\begin{thm}
Functions $(u,v)$ satisfy for any $T>0$ the following properties:
\begin{align}
 &u,v\in C^{1/2,1/8}_{x,t}(\bar{Q}_{T}),\ 
 u,v\in C^{4,1}_{x,t}(P),\label{iass1}\\
 &M_{11}p_{1,x}+M_{12}p_{2,x},\ M_{21}p_{1,x}+M_{22}p_{2,x}\in L^2(P),\label{iass0}\\
& |u|^2p_{1,x} 		 \in L^2(R),
  |v|^2p_{2,x}		 \in L^2(S);
\end{align}
where $P=\bar{Q}_{T}\backslash(\{u=0\}\cup\{v=0\}\cup\{t=0\})$, $R=\bar{Q}_{T}\cap\{v=0\}\cap\{|u|>0\}$ and $S=\bar{Q}_{T}\cap\{u=0\}\cap\{|v|>0\}$.
Furthermore, there exist functions $w_1\in L^2(R)$ and $w_2\in L^2(S)$, such that $(u,v)$ satisfies \rf{mod1} in the following sense:
\begin{align}\label{iass3}
 \iint\limits_{Q_{T}}u\phi_{t}&+\iint\limits_P \left(M_{11} p_{1,x}+M_{12}p_{2,x}\right)\phi_{x} \\&\quad\quad+\iint\limits_{R}\left(|u|^2p_{1,x}+|u|w_1\right)\phi_{x}=0,\nonumber\\
  \iint\limits_{Q_{T}}v\phi_{t}&+\iint\limits_P \left(M_{21} p_{1,x}+M_{22}p_{2,x}\right)\phi_{x} +\iint\limits_{S}(|v|w_2+(1+\alpha)|v|^2p_{2,x})\phi_{x}=0
 \label{iass4}
\end{align}
for all $\phi\in Lip(\bar{Q}_{T}),\phi=0$ near $t=0$ and $t=T$;
\begin{align}
&u(x,0)=u_{0}(x),\ v(x,0)=v_{0}(x),\quad x\in\bar{\Omega},\label{iass2}\\
&||u(\cdot,t)||_{L^1(\Omega)}=||u_0||_{L^1(\Omega)},\ ||v(\cdot,t)||_{L^1(\Omega)}=||u_0||_{L^1(\Omega)},\label{iass6}\\
&u_{x}(\cdot,t)\rightarrow u_{0,x}\text{ and }v_{x}(\cdot,t)\rightarrow v_{0,x}\text{ strongly in }L^2(\Omega)\text{ as }t\rightarrow0,\label{iass5}
\end{align}
and
\begin{equation}
 \begin{split}
  \label{iweak8}
u\text{ and }v &\text{ satisfy \rf{basbc} at all points of the lateral}\\ &\text{ boundary, where }u\neq 0\ \text{and}\ v\neq 0.
 \end{split}
\end{equation}
Finally, the following energy inequality holds
\begin{align}\label{iEE}
E(u(\cdot,T),v(\cdot,T))+\iint\limits_P \left(M_{21} p_{1,x}^2+M_{22}p_{2,x}^2\right)\le E(u_0,v_0).
\end{align}
\end{thm}
\begin{proof}
The assertions \rf{iass1}--\rf{iass0} and \rf{iass2}--\rf{iass5} follow exactly as in the proof of Theorem 2.1. Using \rf{impbound} one observes
\begin{align*}
 \iint\limits_{Q_T}(M_{11\eps}p_{1,\eps,x}^2+M_{22\eps}p_{2\eps,x}^2)dxdt&\leq C -2\iint\limits_{Q_T} M_{12\eps} p_{1\eps,x}p_{2\eps,x}dxdt\\
&= C - 2\iint\limits_{Q_T}(|u_\eps||v_\eps| p_{1\eps,x}p_{2\eps,x})dxdt
\end{align*}
By Young's inequality
\begin{align*}
2|u_\eps||v_\eps| p_{1\eps,x}p_{2\eps,x}) \leq\frac{2}{2+\alpha}|u_\eps|^2 p_{1\eps,x}^2+\left(1+\frac{\alpha}{2}\right)|v_\eps|^2p_{2\eps,x}^2
\end{align*}
and hence one obtains
\begin{align}\label{p1p2i}
&\int\limits_{Q_T}|u_\eps|^2p_{1\eps,x}^2dx\leq C,&\ \int\limits_{Q_T}|v_\eps|^2p_{2\eps,x}^2dx\leq C,\\
&\iint\limits_{Q_T}(M_{11\eps}p_{1,\eps,x}^2+M_{22\eps}p_{2\eps,x}^2)dxdt\leq C.\nonumber
\end{align}
From the last inequality \rf{iEE} then follows.

Next, for $\phi$ as in \rf{iass4}--\rf{iass4} one writes again \rf{weak9}--\rf{weak10}.
Considering the set $I_{1,\delta}$ and $I_{2,\delta}$ as in the proof of Theorem 2.1 and using the estimates \rf{p1p2i}
one can show in the analogous manner that \rf{weak9} and \rf{weak10} converge up to a diagonal subsequence as $\delta\rightarrow 0$ and $\eps\rightarrow 0$
to \rf{iass3} and \rf{iass4} respectively.
\end{proof}
\begin{thm}
The global weak solutions $u$ and $v$ constructed in the previous theorem are nonnegative.
\end{thm}
\begin{proof}
The proof proceeds similarly to  the argument for nonnegativity of the weak solutions in the no-slip case presented in the section 3.1.
The following estimates
\begin{align*}
 (G_{\eps,2}''(u_{\eps}))^2(M_{11\eps}+\eps)\le C,\ \text{and}\ (G_{\eps,2}''(v_{\eps}))^2(M_{22\eps}+\eps)\le C,
\end{align*}
and
\begin{align*}
 &\frac{d}{dt}\int\limits_\Omega \left(G_{\eps,3}(u_\eps)+G_{\eps,2}(v_\eps)\right)dx+\frac{1}{2}\frac{d}{dt}E(u_\eps,v_\eps)\\
&\quad+(1-\delta)\int\limits_\Omega \left(M_{11\eps}+\eps)p_{1\eps,x}^2+2 M_{12\eps}p_{1\eps,x}p_{2\eps,x}+(M_{22\eps}+\eps)p_{2\eps,x}^2\right)dx\\
&\leq \frac{1}{2\delta}\int\limits_\Omega \left(u_{\eps,x}^2+v_{\eps,x}^2 \right)dx.
\end{align*}
are used as analogs to \rf{GC}--\rf{GC1}  in this case in order to obtain the crucial estimate
\begin{align*}
 \int\limits_\Omega \left(G_{\eps,2}(u_{\eps})+G_{\eps,2}(v_{\eps})\right)dx\leq C.
\end{align*}
The rest of the proof proceeds proceeds exactly as in the last paragraph of the section 3.1 but now using \rf{G5} with $n=2$ for both functions $u$ and $v$.
\end{proof}
\section{Conclusion and discussions}

In this article we showed existence of nonnegative global weak solutions for 
the coupled lubrication systems corresponding to the cases of no-slip and Navier-slip conditions at both liquid-liquid and liquid-solid interfaces.
Our results can be generalised in a straight forward way to the system \rf{mod1} considered with the mobility matrix
\begin{align}\label{Mws}
M=\frac{1}{\mu}\begin{pmatrix}
	  \frac{1}{3}u^3+b_1u^2 && \frac{1}{2}u^2v+b_1uv \\
	  \ && \ \\
	  \frac{1}{2}u^2v+b_1uv && \frac{\mu}{3}v^3+u\ \!v^2+b_1v^2+b(\mu+1)v^2
        \end{pmatrix}.
\end{align}
corresponding to the weak-slip conditions at the both interfaces. As it was shown recently in~\cite{JMPW11} the latter model incorporates
both the no-slip and the Navier-slip models \rf{Mns},\rf{Mis} as limiting cases as the slip lengths $b,b_1$ tend simultaneously to zero or infinity, respectively.

One needs to point out that we obtained a slight difference between the weak formulations in the no-slip and Navier-slip cases (compare Theorems 2.1 and 4.1).
Due to the fact that $M_{11}$ and $M_{22}$ components in \rf{Mis} depend only on $u$ or $v$, respectively, in contrast to the no-slip case \rf{Mns}
there is no an analog of estimate \rf{misc18} in the Navier-slip case. Therefore, an additional (so far not identified in terms of solutions $u$ and $v$)
function $w_2$ appears on the singular set $S$ in the latter case. The same problem persists also in the weak-slip case because the leading orders of the mobility
matrix components on the set $S$ coincide with those for the Navier-slip case. 

In this sense we have ``more regularity`` for the weak solutions in the no-slip case then for ones the in weak- or Navier-slip cases. 
This interesting observation should be understood better in future in view of the fact that for the single lubrication equation \rf{NSM} the no-slip case is known to be more singular from both physical and analytical points of view then the slip cases.
At the same time we've become aware of an alternative proof for the existence of weak solutions in the no-slip case in~\cite{EM12} for which the authors have shown
the same regularity as we in the Navier-slip case. But also in our weak formulation for the no-slip case remains an open question weather not yet identified in terms of the solutions function $w_1$ vanishes on the singular set $R$.
Another observation appearing as well due to different component structures of the mobility matrices \rf{Mns} and \rf{Mis} is that we have
stronger entropy for $u$ then for $v$ in the case of the former matrix whereas the entropies are the same for the latter one.

Additionally in contrast to the existing results for the single lubrication equation \rf{NSM} (see e.g.~\cite{bernis1990higher,BP98})
we are not aware if the constructed weak solutions for the systems \rf{mod1} should necessarily posses zero contact angles.
This is due to an absence so far of the strict entropy dissipation inequality for the two layered systems \rf{mod1} which was shown before to hold for \rf{NSM}.
Combined energy-entropy dissipation inequalities derived here (as e.g. \rf{GC1}) do not imply $H^2$ a priori estimates on the solutions.

\section*{Acknowledgements}
SJ is grateful for the support by the  DFG of the project within the priority programme SPP 1506 ``Transport at Fluidic Interfaces''. 
The work of GK was supported by the postdoctoral scholarship at the Max-Planck-Institute for
Mathematics in the Natural Sciences, Leipzig. SJ and GK would like to thank Andreas M\"unch, Dirk Peschka and Barbara Wagner for fruitful discussions. The research of RT leading to these results has received funding from the European Community's Seventh Framework Programme FP7/2007-2013  under Grant Agreement no PIIF-GA-2009-25452--TFE.


\end{document}